\documentclass[12pt]{amsart}

\usepackage{latexsym}
\usepackage{amssymb}
\usepackage{mathrsfs}
\usepackage{amsmath}
\usepackage{fancybox,color}
\usepackage{enumerate}

\usepackage{esint} %Added by Karl

\usepackage[active]{srcltx}
\usepackage[colorlinks]{hyperref}
 \newcommand{\eps}{{\varepsilon}}
  
 \def\1{\raisebox{2pt}{\rm{$\chi$}}}
 
\def\vint_#1{\mathchoice%
          {\mathop{\kern 0.2em\vrule width 0.6em height 0.69678ex depth -0.58065ex
                  \kern -0.8em \intop}\nolimits_{\kern -0.4em#1}}%
          {\mathop{\kern 0.1em\vrule width 0.5em height 0.69678ex depth -0.60387ex
                  \kern -0.6em \intop}\nolimits_{#1}}%
          {\mathop{\kern 0.1em\vrule width 0.5em height 0.69678ex
              depth -0.60387ex
                  \kern -0.6em \intop}\nolimits_{#1}}%
          {\mathop{\kern 0.1em\vrule width 0.5em height 0.69678ex depth -0.60387ex
                  \kern -0.6em \intop}\nolimits_{#1}}}
\def\vintslides_#1{\mathchoice%
          {\mathop{\kern 0.1em\vrule width 0.5em height 0.697ex depth -0.581ex
                  \kern -0.6em \intop}\nolimits_{\kern -0.4em#1}}%
          {\mathop{\kern 0.1em\vrule width 0.3em height 0.697ex depth -0.604ex
                  \kern -0.4em \intop}\nolimits_{#1}}%
          {\mathop{\kern 0.1em\vrule width 0.3em height 0.697ex depth -0.604ex
                  \kern -0.4em \intop}\nolimits_{#1}}%
          {\mathop{\kern 0.1em\vrule width 0.3em height 0.697ex depth -0.604ex
                  \kern -0.4em \intop}\nolimits_{#1}}}

\newcommand{\intav}{\vint}
\newcommand{\aveint}[2]{\mathchoice%
          {\mathop{\kern 0.2em\vrule width 0.6em height 0.69678ex depth -0.58065ex
                  \kern -0.8em \intop}\nolimits_{\kern -0.45em#1}^{#2}}%
          {\mathop{\kern 0.1em\vrule width 0.5em height 0.69678ex depth -0.60387ex
                  \kern -0.6em \intop}\nolimits_{#1}^{#2}}%
          {\mathop{\kern 0.1em\vrule width 0.5em height 0.69678ex depth -0.60387ex
                  \kern -0.6em \intop}\nolimits_{#1}^{#2}}%
          {\mathop{\kern 0.1em\vrule width 0.5em height 0.69678ex depth -0.60387ex
                  \kern -0.6em \intop}\nolimits_{#1}^{#2}}}

\theoremstyle{plain}
  \newtheorem{theorem}{Theorem}[section]
\newtheorem{lemma}{Lemma}[section]
 
 \newtheorem{corollary}{Corollary}[section]
 \newtheorem{propo}{Proposition}[section]
  \newtheorem{claim}{Claim}[section]

\theoremstyle{remark}

\newcommand{\xintloo}[1]{\int\limits_{#1} \kern-18pt\raise4pt\hbox to7pt {\hrulefill}\ }   \numberwithin{equation}{section}

\newcommand{\half}{{\frac{1}{2}}}

\newcommand{\oO}{\overline{\Omega}}

\newcommand{\dist}{\operatorname{dist}}

\newcommand{\D}{\mathcal{D}}
\newcommand{\M}{\mathcal{M}}
\newcommand{\To}{\mathcal{T}}

\DeclareMathOperator{\tr}{tr}

\begin{document}
 \title[Discrete Stochastic Interpretation of $\D_{p}$]{A Discrete Stochastic Interpretation of the Dominative $p$-Laplacian}
 \author[K. Brustad]{Karl K. Brustad}
 \address{Karl K. Brustad
 \hfill\break\indent
Department of Mathematics and System Analysis
\hfill\break\indent
Aalto University
\hfill\break\indent FI-00076, Aalto, Finland
\hfill\break\indent
{\tt karl.brustad@aalto.fi}}
\
 \author[P. Lindqvist]{Peter Lindqvist}
 \address{Peter Lindqvist
 \hfill\break\indent
Department of Mathematics
\hfill\break\indent
Norwegian University of Science and Technology
\hfill\break\indent NO-7491, Trondheim, Norway
\hfill\break\indent
{\tt peter.lindqvist@ntnu.no}}
 \author[J. Manfredi]{Juan J. Manfredi}
 \address{Juan J. Manfredi
 \hfill\break\indent
Department of Mathematics
\hfill\break\indent
University of Pittsburgh
\hfill\break\indent Pittsburgh, PA 15260, USA
\hfill\break\indent
{\tt manfredi@pitt.edu}}
 %\author[M. Parviainen, University of Jyv\"askyl\"a]{MIkko Parviainen,
 %University of Jyv\"askyl\"a}
 %\author[F\'elix del Teso, Trondheim University]{F\'elix del Teso, Trondheim University}
 \numberwithin{equation}{section}
 \date{\today}
%\keywords{Dirichlet boundary conditions, Dynamic Programming Principle, $p$-Laplacian,
 %stochastic games,
% two-player zero-sum games, uniqueness.}
%\subjclass[2010]{35J92, 35D40, 35J87, 49L20, 49L25, 60H30, 91A15, 91A44,
%35B50,35B51, 35J60, 35R02, 35R03, 35B65, 35D40, 35K55, 35B05, 35J62}
\maketitle
\setcounter{tocdepth}{1}
%level -1: part, 0: chapter, 1: section, etc.
\tableofcontents
%Results from prior support are described in Section \S \ref{prior} below followed by  Broader Impacts in  %Section \S \ref{bi}. 

\section{Introduction}\label{introduction}
The interplay between Stochastic Game Theory and \emph{nonlinear} Partial Differential Equations has been of increasing importance, beginning with the pioneering work of Kohn and Serfaty
\cite{KS09, KS10} and Peres, Schramm, Sheffield and Wilson \cite{PS08, PSSW09}, involving discrete processes. We shall develop this connection for the \emph{Dominative p-Laplace Equation}, which is akin to the well-known normalized p-Laplace Equation. Thus, we shall present a discrete stochastic interpretation and prove uniform convergence of the discretizations.

The Dominative $p$-Laplacian is the operator defined for $2\le p < \infty$ as follows:
\begin{equation}\label{dominativep}
\D_p u(x):=\lambda_{1}+\cdots+\lambda_{N-1}+(p-1)\lambda_{N}, 
 \end{equation}
where we have  ordered the eigenvalues of the Hessian matrix $D^{2}u(x)$ as $\lambda_{1}\le \lambda_{2}\ldots\le\lambda_{N}$.
It was introduced by Brustad in \cite{B18} in order to give a  natural explanation of the superposition principle for the $p$-Laplace equation (see \cite{CZ03} and \cite{LM08}). This operator is interesting in its own right. The case $p=2$ reduces to the Laplacian $\Delta u(x)$.
\par
It is also of interest to consider the case $p=\infty$ with the following interpretation
\begin{equation}\label{dominativeinf}
\D_{\infty}u(x):=\lambda_{N}.
\end{equation}
$\D_\infty u = 0$ is the largest eigenvalue equation, or the equation for the concave envelope, which has been studied in \cite{O07} and \cite{OS11}. This equation is only \emph{degenerate} elliptic, but viscosity solutions with $C^{1, \alpha}$ boundary values are in the class $C^{1, \alpha}$ \cite{OS11}. \par
%\end{remark}

The operator $\D_{p}$ is sublinear, therefore convex,  and uniformly elliptic for $p<\infty$. Thus,  the viscosity solutions of the equation $\D_{p}u(x) = 0$ are locally in the class  $C^{2, \alpha}$. See Chapter 6 in \cite{CC95} for the regularity result and  \cite{B18-2} for the general theory of sublinear operators. 
\par

Consider the following problem.
Suppose we have a domain $\Omega$ and a function $F$ defined on an $\epsilon$-strip along the \emph{outside} of its boundary.  Start at a point $x_0$ in $\oO$. Now, you get to choose a \emph{direction} $\sigma = \sigma(x_0)\in\mathbb{S}^{N-1}$. Then, for a fixed $q\in(0,1]$, a new starting point $x_1\in\overline{B}_\eps(x_0)$ is picked according to the rule
\begin{equation}\label{rule}
\begin{cases}
&\text{with probability $q$, $x_{1}\in B_{\epsilon}(x_{0})$ is selected at random}\\
&\text{with probability $\frac{1-q}{2}$ we set $x_{1} := x_{0}+\epsilon\sigma$, and}\\
&\text{with probability $\frac{1-q}{2}$ we set $x_{1} := x_{0}-\epsilon\sigma$.}
\end{cases}
\end{equation}
%\begin{enumerate}
%	\item with probability $q$, $x_{1}\in B_{\epsilon}(x_{0})$ is selected at random,\label{transition}
%	\item with probability $\frac{1-q}{2}$ we set $x_{1} := x_{0}+\epsilon\sigma$, and
%	\item with probability $\frac{1-q}{2}$ we set $x_{1} := x_{0}-\epsilon\sigma$.
%\end{enumerate}
Observe that the probabilities sum up to 1, as they should. Also note that $x_1$ does not depend on the sign of $\sigma$. Thus you are effectively only choosing a \emph{line} through the origin.
Repeat the process until you leave $\oO$ at, say, step $\tau$.
The value $F(x_\tau)$ is then defined and let us say you want it to be as large as possible. The challenge is therefore:
How to choose the directions $(\sigma(x_k))_{k=0}^{\tau-1}$ in order to maximize the expected value of $F(x_\tau)$?

We shall show that the maximized expected value $u_\eps(x) := \sup_\sigma\mathbb{E}(F(x_\tau))$, obtained from starting at $x\in\oO$, satisfies a \emph{non-linear mean value property}, or Dynamic Programming Principle,
\begin{equation*}
\displaystyle{u_{\epsilon}(x)=
	q \intav_{B_{\epsilon}(x)} u_{\epsilon}(y)\, dy+\!\!
	\left(1-q\right)\sup_{\vert\xi\vert = 1}
	\frac{u_{\epsilon}(x+\epsilon\xi) + u_{\epsilon}(x-\epsilon\xi)}{2}}, 
\end{equation*}
where we have used the notation 
$$ \intav_{B} f(y)\, dy= \frac{1}{|B |}\int_{B} f(y)\, dy$$
for the average of an integrable  function $f$ on a ball $B$.\par
Moreover, we will prove that $u_{\epsilon}\to u$ uniformly 
in $\overline{\Omega}$, where the limit function $u$ is the unique solution of the Dirichlet problem
\begin{equation}\label{introdirichlet}
\left\{
\begin{array}{rclc}
\D_{p}u(x) & = & 0 & \text{ for }x\in\Omega\\
u(x)& = & F(x) & \text{ for }x\in\partial\Omega\\
\end{array}
\right.
\end{equation}
for the Dominative $p$-Laplace Equation.
The relation between $p\in[2,\infty)$ and $q$ is 
\begin{equation}\label{pq}
q = \frac{N+2}{N+p},\qquad 1-q = \frac{p-2}{N+p}.
\end{equation}
We shall assume that $\Omega\subset\mathbb{R}^{N}$ is a bounded Lipschitz domain and that the prescribed boundary values
$F\colon\partial\Omega\mapsto\mathbb{R}$
are Lipschitz continuous.

%In this paper, we fix $p\in[2,\infty)$ and present a discrete stochastic approximation to the unique viscosity solution of the Dirichlet  problem for the Dominative $p$-Laplace Equation
%\begin{equation}\label{introdirichlet}
%\left\{
%\begin{array}{rclc}
%\D_{p}u(x) & = & 0 & \text{ for }x\in\Omega\\
%u(x)& = & F(x) & \text{ for }x\in\partial\Omega\\
%\end{array}
%\right.
%\end{equation}
%in a  bounded Lipschitz domain $\Omega\subset\mathbb{R}^{N}$. The  prescribed boundary values
%$F\colon\partial\Omega\mapsto\mathbb{R}$
%are assumed to be Lipschitz continuous.
%\par
%For $\epsilon>0$ we construct an approximation $u_{\epsilon}$ that satisfies a \emph{non-linear mean value property}, or Dynamic Programming Principle,
%\begin{equation*}
%\displaystyle{u_{\epsilon}(x)=
%	q \intav_{B_{\epsilon}(x)} u_{\epsilon}(y)\, dy+\!\!
%	\left(1-q\right)\sup_{\vert\xi\vert = 1}
%	\frac{u_{\epsilon}(x+\epsilon\xi) + u_{\epsilon}(x-\epsilon\xi)}{2}},
%\end{equation*}
%where $q\in(0,1]$ is
%\color{blue}
%\begin{equation}\label{pq}
%q:=\frac{N+2}{N+p}.
%\end{equation}
%\normalcolor
%%and  $\sigma$ is a  \textsf{Borel} function  $\sigma\colon\Omega\mapsto\mathbb{S}^{N-1}$, which we call a \emph{control}. 
%%\par
%We give a game-theoretic interpretation of $u_{\epsilon}$ and show that $u_{\epsilon}\to u$ uniformly 
%in $\overline{\Omega}$, where the limit function $u$ is the unique solution of the Dirichlet problem (\ref{introdirichlet}). \par
%%\begin{remark}

\section{Statements of Results}  
Let $\Omega$ be a bounded Lipschitz domain in $\mathbb{R}^{N}$, $N\ge 2$.
For a fixed  $\eps>0$ we set
$$\Omega_{\epsilon} := \{x\in \mathbb{R}^{N}\colon d(x,\Omega)\le\epsilon
\}\text{ and  } \Gamma_{\epsilon} := \Omega_{\epsilon}\setminus \overline{\Omega}.$$

%\color{red}(I changed the definition of $\Gamma_\eps$ so that it doesn't contain the boundary $\partial\Omega$. I think this is more natural)\normalcolor

Note that, by our definition, $\Gamma_\eps$ does not contain the boundary $\partial\Omega$. Also, for $x\in \overline{\Omega}$, we  always have $\overline{B}_{\epsilon}(x)\subset \Omega_{\epsilon}$.
We may extend the given  bounded Lipschitz function
$F\colon\partial\Omega\mapsto\mathbb{R}$ to $\Gamma_{\epsilon}$, preserving the same Lipschitz constant.

Fix $p\in[2,\infty)$  and recall the relation \eqref{pq} with $q$. We define a \textit{non-linear} Mean Value Operator $\M^\eps$  from the set of bounded and integrable Borel functions in $\Omega_\eps$, to the set of bounded functions in $\overline{\Omega}$ 
%$\M^\eps\colon \mathcal{A}\to \{\text{ Bounded functions in  } \overline{\Omega}\ \}$
as 
 $$\M^\eps v(x) :=
 q \fint_{B_{\epsilon}(x)}v(y)\, dy+
 (1-q)\sup_{|\xi|=1}
 \frac{v(x+\epsilon\xi)+ v(x-\epsilon\xi)}{2}.
 $$ 
For $\xi\in\mathbb{S}^{N-1}$ we shall also write
\[\M_\xi^\eps v(x) := 
q \fint_{B_{\epsilon}(x)}v(y)\, dy+
(1-q)
\frac{v(x+\epsilon\xi)+ v(x-\epsilon\xi)}{2}\]
and thus $\M^\eps v = \sup_{|\xi|=1}\M_\xi^\eps v$.
If $v_1,v_2$ are two such functions in $\Omega_\eps$ and $c$ is a non-negative constant, one may easily check that

\begin{enumerate}
\item $\M^\eps[cv_1](x) = c\M^\eps v_1(x)$,
\item $\M^\eps[v_1 + v_2](x) \leq \M^\eps v_1(x) + \M^\eps v_2(x)$,
\item $\M^\eps[v_1 - v_2](x) \geq \M^\eps v_1(x) - \M^\eps v_2(x)$,
\item $ \M^\eps v_1(x) \leq \M^\eps v_2(x)$ whenever $v_{1}\le v_{2}$ in $\overline{B}_{\epsilon}(x)$,
\end{enumerate}
for $x\in\oO$. Moreover, $\M^\eps[a+v_1] = a + \M^\eps v_1$ in $\overline{\Omega}$ for any affine function $a$ in $\Omega_\eps$, and $\M^\eps[\phi + v_1] = \M^\eps\phi + \M^\eps v_1 = \M^\eps_\xi\phi + \M^\eps v_1$ for any paraboloid $\phi(x) = \alpha\vert x-x_0\vert^2$ and $\xi\in\mathbb{S}^{N-1}$. 
Also, $\M^\eps$ is translational invariant, meaning that if $\theta_h(x) := x+h$, then $\M^\eps[v\circ\theta_h] = (\M^\eps v)\circ\theta_h$ in the proper domains.

Next, we show that
\begin{equation}\label{quadid}
C_{N,p}\frac{\M^\eps \phi(x) - \phi(x)}{\epsilon^{2}} = \D_p \phi,\qquad C_{N,p} := 2(N+p),
\end{equation}
%is constant and independent of $\eps$
for second order polynomials $\phi$.

Let $\phi(x) := c + b^Tx + \half x^TAx$ in $\mathbb{R}^N$, where $A$ is a symmetric $n\times n$ matrix. First, we verify \eqref{quadid} at $x=0$.
\begin{align*}
\M^\eps\phi(0) - \phi(0)
&= q\fint_{B_{\epsilon}(0)} \half y^TAy\, dy + \frac{1-q}{2}\eps^2\sup_{|\xi|=1}\xi^TA\xi\\
&= q\frac{\eps^2}{2(N+2)}\tr A + \frac{1-q}{2}\eps^2\lambda_{N}(A)\\
&= \eps^2\left(\frac{q}{2(N+2)}\Delta\phi + \frac{1-q}{2}\lambda_{N}(D^2\phi)\right)\\
&= \frac{\eps^2}{2(N+p)}\left(\Delta\phi + (p-2)\lambda_{N}(D^2\phi)\right)\\
&= \frac{\eps^2}{2(N+p)}\D_p\phi(0).
\end{align*}
Next, for $h\in\mathbb{R}^n$ we have $\phi(x+h) - \phi(h) = h^TAx + \phi(x)-\phi(0)$ and so $\M^\eps\phi(x+h) - \phi(h) = h^TAx + \M^\eps\phi(x) - \phi(0)$ by translational invariance and the property for affine functions. The identity \eqref{quadid} follows by setting $x=0$ and replacing $h$ with $x$.

In particular, by a Taylor expansion about $x\in\Omega$, \eqref{quadid} implies that
\begin{equation}\label{smoothexpansion}
C_{N,p}\frac{\M^\eps v(x) - v(x)}{\epsilon^{2}}=\D_p v(x) + O(\eps),\qquad \text{as $\eps\to 0$,}
\end{equation}
for every $v\in C^3(\oO)$.\normalcolor

Let  $\mathcal{A}$ denote the class of  Borel functions  $v\colon \Omega_{\epsilon}\to \mathbb{R}$ satisfying the conditions
\begin{enumerate}
	\item $v\in L^\infty(\Omega_{\epsilon})$,
	%\item  lower semicontinuous (l.s.c) in $X$
	and 
	\item $v=F$ on $\Gamma_{\epsilon}$.
\end{enumerate} 

The next Lemma allows us to  circumvent  the question whether $\M^\eps v$ is measurable.
 
\begin{lemma}\label{lsc}If $v$ is bounded and lower semicontinuous (l.s.c.) in $\Omega_{\epsilon}$, then
$\M^\eps v$ is bounded and l.s.c. in $\overline{\Omega}$.
\end{lemma} 
\vskip -.15in
(See \S \ref{proofof21} below  for the proof.)\par

Recall that for a bounded function $v$ the lower semi-continuous envelope $v_{*}$  is given by
$$ v_{*}(x) := \sup \{ \phi(x) \colon \phi \le v \text{  and  } \phi \text{ is l.s.c.}\}.$$
Define the iteration operator $\To^\eps:\mathcal{A}\mapsto\mathcal{A}$ as follows

\begin{equation}\label{operatorTsigma}
\left\{
\begin{array}{cccl}
\text{for } x\in \overline{\Omega}, & \To^\eps v(x) & := &\M^\eps [v_{*}](x) \\
\text{for } x\in \Gamma_{\epsilon},  &  \To^\eps v(x) &= & F(x).
\end{array}
\right.
\end{equation}

Its fixed point is of interest.
\begin{lemma}\label{existenceanduqineness}
There exists a unique  function $v_{\epsilon}\in \mathcal{A}$ such that  $\To^\eps v_{\epsilon}(x)=v_{\epsilon}(x)$ for all $x\in \Omega_\eps
$.
Moreover,  the function $v_\eps$ is l.s.c. in $\overline{\Omega}$.
\end{lemma}
\vskip -.15in
(See \S \ref{proofof22} below  for the proof.)\par
We keep the subindex $\epsilon$ to emphasize the dependence on the step-size.
We call  $v_{\epsilon}$ the \textit{$\epsilon$-mean value solution}. 
\par

Given a fixed Borel measurable control  $\sigma\colon\oO\mapsto\mathbb{S}^{N-1}$  and a stepsize $\eps>0$, we define a discrete random process according to the rule \eqref{rule}.

More precisely,  fix $x_{0}\in \Omega_\eps$ and let
$$X^{\infty, x_{0}} := \{ \omega=(x_{0},x_{1}, x_{2},\ldots)\colon x_{n}\in \Omega_\eps\}$$
be the space of possible outcomes.
Set $\mathcal{F}^{x_{0}}_0$ to be the trivial sigma-algebra $\{X^{\infty,x_0},\emptyset\}$, and for $n\ge1$ let $\mathcal{F}^{x_{0}}_{n}$ be the sigma-algebra generated by the cylinders
\[
\begin{array}{rl}
A_{1}\times A_{2}\times\cdots &\times A_{n}\times \Omega_\eps \times \Omega_\eps\cdots\\
& = \{\omega\in X^{\infty, x_{0}}\colon x_{i}\in A_{i}, i=1,\ldots, n\}\\
&=  A_{1}\times A_{2}\times\cdots \times A_{n} \quad  \textrm{(abuse of notation)}, 
\end{array}
\]
 where the $A_{i}\subset \Omega_\eps$ are Borel sets. \par

Clearly we have 
$\mathcal{F}^{x_{0}}_{n}\subset \mathcal{F}^{x_{0}}_{n+1}$ so that  $\left\{ \mathcal{F}^{x_{0}}_{n} \right\}_{n\ge 1}$ is a filtration of the sigma-algebra $\mathcal{F}^{x_{0}}$ generated by
$$\bigcup_{{n\ge 1}}\mathcal{F}^{x_{0}}_{n}.$$
The coordinate functions $\textbf{x}_{n}(\omega)=x_{n}$ are $\mathcal{F}^{x_{0}}_{n}$ and $\mathcal{F}^{x_{0}}$ measurable.\par
Let $\tau_{\sigma}\colon X^{\infty, x_{0}}\to \mathbb{N}\cup \{\infty\}$ be the random variable
$$\tau_{\sigma}(\omega)=\min\{n\in\mathbb{N}\colon x_{n}\in \Gamma_{\epsilon}\},$$
where we follow the convention $\min\emptyset=\infty$.  We say that $\tau_{\sigma}$ is a \textsf{stopping time} with respect to the filtration $\left\{ \mathcal{F}^{x_{0}}_{n} \right\}_{n\ge 1}$.\par
For  $x\in \Omega_\eps$ define the \emph{transition probability measures} $\gamma[x]$ as
\[\gamma[x](A):=
\begin{cases}
\delta_{x}(A), &\text{if $x\in\Gamma_{\epsilon}$,}\\
q\,
\dfrac{|B_{\epsilon}(x)\cap A|}{|B_{\epsilon}(x)|} + \dfrac{1-q}{2}(\delta_{x+\epsilon \sigma(x)}(A)+\delta_{x-\epsilon \sigma(x)}(A)), &\text{if $x\in\oO$.}
\end{cases}\]

We see that the mapping $x\to\gamma[x](A)$ is Borel measurable for a fixed  Borel set $A\subset X$. Indeed, the first term is, in fact, continuous and the second one is easily seen to be Borel measurable, since $x\to\sigma(x)$ is so.

\par
For $n\ge 1$ define the probability measures $\mathbb{P}^{n,x_{0}}_{\sigma}$ on the measurable space
$(X^{\infty, x_{0}}, \mathcal{F}^{x_{0}}_{n})$ as follows:

$$\mathbb{P}^{1,x_{0}}_{\sigma}(A_{1}) := \gamma[x_0](A_1) = \int_{A_{1}} 1 \, d\gamma[x_{0}] (y_{1}),
$$
(Note that $x_{0}$ is fixed and the integration variable $y_{1}\in A_{1}$. )
$$\mathbb{P}^{2,x_{0}}_{\sigma}(A_{1}\times A_{2}) := \int_{A_{1}}\left(\int_{A_{2}} 1\, d\gamma[y_{1}](y_{2})
\right)d\gamma[x_{0}](y_{1})
$$
In the general case we get
\begin{align*}&\mathbb{P}^{n,x_{0}}_{\sigma}(A_{1}\times\cdots\times A_{n})\\
  &:= \int_{A_{1}}\!\!\left(\int_{A_{2}}\!\!\left(
\cdots\int_{A_{n}} 1\, d\gamma[y_{n-1}](y_{n})
\right) \cdots d\gamma[y_{1}](y_{2})\right)\! d\gamma[x_{0}](y_{1}).
\end{align*}

The family of probabilities $\left\{  \mathbb{P}^{n,x_{0}}_{\sigma}  \right\}_{n\ge1}$  is consistent in the sense of Kolmogorov. Thus the limit probability
$$ \mathbb{P}^{x_{0}}_{\sigma} := \lim_{n\to\infty}\mathbb{P}^{n,x_{0}}_{\sigma}   $$  exists and we have
$$\mathbb{P}^{n,x_{0}}_{\sigma}(A_{1}\times\cdots\times A_{n})=\mathbb{P}^{x_{0}}_{\sigma}(A_{1}\times\cdots\times A_{n})$$ for all cylinders $A_{1}\times\cdots\times A_{n}$.

 The following lemma tells us that the conditional expectation of the process at step $n$ relative to its past history, reflected in the sigma-algebra $\mathcal{F}^{x_{0}}_{n-1}$,  is precisely the integral of $v$ with respect to the transition probability from step $n-1$ to $n$. 
\begin{lemma}\label{key}
Let $v\colon \Omega_\eps\mapsto \mathbb{R}$ be  a bounded measurable  function. Then, whenever $x_{n-1}\in\oO$, we have
\[\mathbb{E}^{x_{0}}_{\sigma}\left[ v(x_{n})\,  | \,  \mathcal{F}^{x_{0}}_{n-1} \right] \!\!(x_{n-1})  =   \M^\eps_{\sigma(x_{n-1})} v(x_{n-1}),\]
and thus
\[\sup_\sigma\mathbb{E}^{x_{0}}_{\sigma}\left[ v(x_{n})\,  | \,  \mathcal{F}^{x_{0}}_{n-1} \right] \!\!(x_{n-1})  =   \M^\eps v(x_{n-1}).\]
\end{lemma}
\vskip -.15in

(See \S \ref{proofof23} below  for the proof.)
\begin{lemma}\label{martingale1}
For any fixed $y_{0}\in \mathbb{R}^{N}$ and every control $\sigma$ the sequence of random variables $$\{ |x_{n\wedge\tau_{\sigma}}-y_{0}|^{2}-c_{N,p}(n\wedge \tau_{\sigma}) \epsilon^{2}\}_{n\ge 1}$$ is a martingale with respect to the natural filtration
$\left\{\mathcal{F}^{x_{0}}_{n}\right\}_{n\ge 1}.$
Here,  $c_{N,p} := \frac{N+p-2}{N+p}$. 
\end{lemma}
 \vskip -.15in
(See \S \ref{proofof24} below  for the proof.)\par
Applying Doob's optional stopping to the  finite stopping times $\tau_{\sigma}\wedge n := \min\{\tau_\sigma,n\}$  and letting 
$n\to\infty$, we have  
\begin{equation}\label{taubound}
\epsilon^{2} \mathbb{E}_{
\sigma}^{x_{0}}[\tau_{\sigma}]\le  C(N,\Omega)
\end{equation} and the
 process ends almost surely: 
 \begin{equation}\label{processends}
 \mathbb{P}^{x_{0}}_{\sigma}( \{\tau_{\sigma} <\infty\})=1.
 \end{equation}

Therefore, when we run the process we will hit  $\Gamma_{\epsilon}$ almost surely. Thus, the random variable
$F(x_{\tau_{\sigma}})$ is well defined. Averaging over all possible runs we get the expected value 
\begin{equation}\label{expectation}
u^{\sigma}_{\epsilon}(x_{0}) := \mathbb{E}^{x_{0}}_{\sigma}[F(x_{\tau_{\sigma}})].
\end{equation}

Optimizing over all strategies we get 
\begin{equation}
\label{udef}u_{\epsilon}(x_{0}) := \sup_{\sigma}\left(u^{\sigma}_{\epsilon}(x_{0})
\right)=\sup_{\sigma}\left( \mathbb{E}^{x_{0}}_{\sigma}[F(x_{\tau_{\sigma}})]
\right),
\end{equation}
which we call the  \textit{$\epsilon$-stochastic solution}.

Recall the $\epsilon$-mean value solution $v_\eps$ defined in Lemma \ref{existenceanduqineness}.

\begin{theorem}\label{meanvalue=stochastic}
We have
\[u_\eps = v_\eps\qquad\text{in $\Omega_\eps$.}\]
That is, the $\eps$-stochastic solution $u_{\epsilon}$ also satisfies the dynamic programming principle  $u_{\epsilon}(x_{0})= \M^\eps u_{\epsilon}(x_{0})$.
\end{theorem}
%\vskip -.15in
(See \S \ref{proofofthm21} below  for the proof.)\par

The following comparison principle for $\epsilon$-mean value solutions follows at once from
formula \eqref{udef} and Theorem \ref{meanvalue=stochastic}.

\begin{lemma}\label{comparisonprinciple}
Let  $v_{\epsilon}$ be the $\epsilon$-mean value solution with boundary values $F$ and
let $w_{\epsilon}$ be the $\epsilon$-mean value solution with boundary values $G$. If
$F\le G$ on $\partial\Omega$ (extended so that we still have $F\le G$ on $\Omega_\eps$), then $v_{\eps}\le w_{\eps}$ in $\oO$. 
\end{lemma}

We  next  adapt the Barles and Perthame procedure as in  \cite{BS91} of semi-continuous regularizations. We remark that in \cite{BS91} the domain $\Omega$ must be of class $C^{2}$ and the equation must satisfy a strong uniqueness property involving the viscosity interpretation of the boundary Dirichlet data.  We replace the strong uniqueness property with uniform boundary estimates for the discretizations  $u_{\epsilon}$ to reach the same uniform convergence conclusion as in \cite{BS91}.

\begin{lemma}\label{boundaryestimate}
	Given $\eta>0$ we can find $\epsilon_{0}>0$  and $\epsilon_1>0$ such that whenever
	 $y_{0}\in\partial\Omega$ and $\epsilon< \epsilon_{0}$ we have
	$$|u_{\epsilon}(x)- F(y_{0}) |\le \eta$$
	for $x\in B_{\epsilon_2}(y_0)\cap\overline{\Omega}$.
\end{lemma}

%\begin{lemma}\label{boundaryestimate}
%Given $\eta>0$ we can find $\delta_{1}>0$ and  $\epsilon_{1}>0$ such that whenever
%$x_{0}\in \Omega$, $y_{0}\in\partial\Omega$, $|x_{0}-y_{0}|<\delta_{1}$ and $\epsilon< \epsilon_{1}$ we have
%$$|u_{\epsilon}(x_{0})- F(y_{0}) |\le \eta.$$
%\end{lemma}
% \vskip -.15in
(See \S \ref{proofof26} below  for the proof.)\par

For $x\in \oO$ and $0< \delta\leq \dist(x,\partial\Omega)$ consider the sets
\begin{equation}\label{setu}
S(x, \delta) := \left\{ u_{\epsilon}(y) \colon \epsilon<\delta  \text{ and }  |y-x| \leq \delta
\right\}.
\end{equation} and the functions
$$U_{\delta}(x) := \sup S(x,\delta).$$
Note that the  set  $\sup S(x,\delta)$ is bounded above,  that 
$S(x, \delta_{1})\subset S(x, \delta_{2})$, and 
\begin{equation}\label{tres}
u_{\epsilon}(x)\le U_{\delta_{1}}(x)\le U_{\delta_{2}}(x)
\end{equation} whenever $\epsilon< \delta_{1}\le \delta_{2}\leq\dist(x,\partial\Omega)$.  
Thus, the following limit is always well-defined
\begin{equation}\label{defubar}
\overline{u} (x) := \lim_{\delta\to0}U_{\delta}(x)
\end{equation}
The function $\overline{u}\colon\Omega\mapsto \mathbb{R}$ is the half-relaxed upper limit of the family $\{u_{\epsilon}\}_{\epsilon}$ when 
$\epsilon\to0$ and it is always u.s.c.\par
Similarly, we consider the functions
$$U^{\delta}(x) := \inf S(x,\delta), $$ so that
\begin{equation}\label{tres-2}
u_{\epsilon}(x)\ge U^{\delta_{1}}(x)\ge U^{\delta_{2}}(x)
\end{equation} whenever $\epsilon< \delta_{1}\le \delta_{2}\leq\dist(x,\partial\Omega)$.  
Thus, the following limit is always well-defined
\begin{equation}\label{defubar}
\underline{u}(x) := \lim_{\delta\to0}U^{\delta}(x)
\end{equation}
The function $\underline{u}\colon\Omega\mapsto \mathbb{R}$ is the half-relaxed lower limit of the family $\{u_{\epsilon}\}_{\epsilon}$ when 
$\epsilon\to0$ and it is always l.s.c.\par

Since $U^\delta\leq U_\delta$ by definition, also $\underline{u}\leq\overline{u}$ in $\oO$. On the boundary $\partial\Omega$, the oposite inequality $\underline{u}\geq\overline{u}$ holds, because
Lemma \ref{boundaryestimate} implies
\begin{equation}\label{upperbound}
\overline{u}(y) \le F(y),
\end{equation}
\begin{equation}\label{lowerbound}
\underline{u}(y) \ge  F(y),
\end{equation}
for every $y\in\partial\Omega$.

 \begin{lemma}\label{viscosity}
$\overline{u}$ is a viscosity subsolution  and $\underline{u}$ is a viscosity supersolution of $ \D_{p}u =0$ in $\Omega$.
 \end{lemma}
 \vskip -.15in
(See \S \ref{proofof27} below  for the proof.)
 \begin{theorem}\label{main}
 We have $\overline{u} =\underline{u}$, denoted by $u$. It is   the unique solution to the Dirichlet problem
(\ref{introdirichlet}). 
 Moreover $u_{\epsilon}\to u$ uniformly in $\overline{\Omega}$
\end{theorem}
  \vskip -.15in
(See \S \ref{proofofthm22} below  for the proof.)

\section{Proofs}

\subsection{  Proof of Lemma \ref{lsc}. (l.s.c. of $\M^\eps v$):}\label{proofof21}

Let $v$ be bounded and l.s.c. in $\Omega_\eps$. That $\M^\eps v$ is bounded in $\oO$, is clear. For $x\in\oO$ we can write
\[\M_\xi^\eps v(x) = \int v(x+y)\,d\gamma_\xi(y)\]
where
\[\gamma_\xi(A) := q\frac{\vert B_\eps(0)\cap A\vert}{\vert B_\eps(0)\vert} + (1-q)\frac{\delta_\xi(A) + \delta_{-\xi}(A)}{2}.\]
Now,
\[\liminf_{x\to x_0}\int v(x+y)\,d\gamma_\xi(y) \geq \int \liminf_{x\to x_0}v(x+y)\,d\gamma_\xi(y)\geq \int v(x_0+y)\,d\gamma_\xi(y)\]
by Fatou's Lemma. Thus $\M_\xi^\eps v$ is l.s.c. in $\oO$ for each $\xi\in\mathbb{S}^{N-1}$, and so is
$\M^\eps v = \sup_{\vert \xi\vert = 1}\M_\xi^\eps v$ being a supremum of l.s.c. functions.
\normalcolor

\subsection{  Proof of Lemma \ref{existenceanduqineness}. (Existence and uniqueness of Mean Value solutions):}\label{proofof22}

We note that $\M^\eps[v_{*}](x)$ is well-defined in $\overline{\Omega}$ for bounded functions $v$ in $\Omega_\eps$.
Write $m_F := \inf_{\Gamma_\epsilon} F$  and $M_F := \sup_{\Gamma_\epsilon} F$. It is easily checked that
$$m_F \le u \le v\le M_F \implies m_F \le   \To^\eps u \le   \To^\eps v\le M_F. $$
That is, $\To^\eps$  is monotone and bounded. Therefore, given an initial function
$m_F \le v_0 \le  M_F$ in $\Omega_\eps$ we get a bounded sequence of functions
$$v_i := \To^\eps v_{i-1}, \ \ i = 1,2,\dots,$$
which, by induction, is monotone provided $v_1 \ge  v_0$  or $v_1 \le  v_0$. In particular, choosing $v_0(x)\equiv m_F$ yields an increasing sequence whose pointwise limit
$$v_\epsilon(x) := \lim_{i\to\infty} v_i(x).$$
satisfies $v_{\epsilon} |_{\Gamma_{\epsilon}} = F$. In $\overline\Omega$  we get
$v_\epsilon=  \M^\eps [(v_\epsilon)_*] = \M^\eps v_\epsilon$ since a point-wise increasing limit of l.s.c. functions is l.s.c. That is, $\To^\eps v_\eps = v_\eps$ in $\Omega_\eps$.\par
   
Suppose that we have two solutions 
$u$ and $w$, and assume for the sake of contradiction that
$$M=\sup_{\Omega_\eps}(u-w)>0.$$
Choose a sequence $(x_{n})\in \Omega_\eps$ such that $\lim_{n\to\infty}(u(x_{n})-w(x_{n}))=M$. Note that indeed $x_{n}\in\oO$. 
We have
\begin{align*}
u(x_{n})-w(x_{n}) &= \M^\eps u(x_{n})- \M^\eps w(x_{n})\\
	&\leq \M^\eps[u-w](x_n)\\
    &\le q\intav_{B_{\epsilon}(x_{n})} (u(y)-w(y))\,dy + (1-q)M
\end{align*}

Let $x_{n}\to x_{0}\in X$ and simplify to get
$$M\le \intav_{B_{\epsilon}(x_{0})}  (u(y)-w(y))\,dy$$
since $u$ and $w$ are bounded and integrable, and by the continuity of the ball measure.
We conclude that $(u-w)(x)=M$ for a.e. $x\in B_{\epsilon}(x_{0})$.  Note that this implies
$x_{0}\in \Omega$  and also that $B_{\epsilon}(x_{0})\subset \Omega$. Define the set 
$$G=\{ x\in \Omega \colon (u-w)=M \textrm{ a.e. in a neighborhood of }x \}.$$
We have shown that $G\not=\emptyset$. The same proof shows that $G$ is closed, and since it is
clearly open, we have $G=\Omega$ so that $(u-w)(x)=M $ a.e. in $\Omega$. \par
To reach a contradiction, take
$y\in\partial\Omega$ and choose $x_{n}\in\Omega$ such that $x_{n}\to y$ and $(u-w)(x_{n})=M$. 
\qed

\subsection{Proof of Lemma \ref{key}}\label{proofof23}

\begin{proof} The conditional expectation  $\mathbb{E}^{x_{0}}_{\sigma}\left[ v\circ \mathbf{x_{n}}\,|\, \mathcal{F}^{x_{0}}_{n-1})
	\right]$ is $\mathcal{F}^{x_{0}}_{n-1}$ measurable, and thus a function of $(x_{1}\dots,x_{n-1})$  such that
	$$\mathbb{E}^{x_{0}}_{\sigma}\left[ \chi_{A}\,
	\mathbb{E}^{x_{0}}_{\sigma}\left[ v\circ \mathbf{x_{n}}\,|\, \mathcal{F}^{x_{0}}_{n-1})
	\right] \right]=\mathbb{E}^{x_{0}}_{\sigma}\left[ \chi_{A}\,(v\circ \mathbf{x_{n}})
	\right]
	$$ for every cylinder $A=A_{1}\times\cdots\times A_{n-1}$.
	Do this for $v=\chi_{B}$, then for simple functions, and then use the monotone convergence theorem. \par
	Temporarily, set $ G(x_{1},\dots, {x_{n-1}})=\mathbb{E}^{x_{0}}_{\sigma}\left[ \chi_{B}\circ \mathbf{x_{n}}\,|\, \mathcal{F}^{x_{0}}_{n-1})
	\right]$. This function must satisfy
	$$\mathbb{E}^{x_{0}}_{\sigma}\left[ \chi_{A}\,G(x_{1},\dots, {x_{n-1}})
	\right]=\mathbb{E}^{x_{0}}_{\sigma}\left[ \chi_{A}\,(\chi_{B}\circ \mathbf{x_{n}})
	\right]
	$$ for every cylinder $A=A_{1}\times\cdots\times A_{n-1}$. We have 
	$$\mathbb{E}^{x_{0}}_{\sigma}\left[ \chi_{A}\,(\chi_{B}\circ \mathbf{x_{n}})
	\right]=\mathbb{P}^{x_{0}}_{\sigma}(A_{1}\times\cdots\times A_{n-1}\times B)$$
	and
	\begin{align*}
	&\mathbb{P}^{n, x_{0}}_{\sigma} (A_{1}\times\cdots\times A_{n-1}\times B)\\
	& =\int_{A_{1}}\left(\int_{A_{2}}\left(
	\cdots\int_{B} 1\, d\gamma[y_{n-1}](y_{n})
	 \cdots \right) d\gamma[y_{1}](y_{2})\right) d\gamma[x_{0}](y_{1})\\
	&=\!\int_{A_{1}}\!\biggl(\int_{A_{2}}\!\!\biggl(
	\cdots  q\,
	\frac{|B_{\epsilon}(y_{n-1})\cap B|}{|B_{\epsilon}(y_{n-1})|}+\frac{1-q}{2}\bigl(\delta_{y_{n-1}+\epsilon \sigma(y_{n-1})}(B)\\
	&\qquad\qquad\qquad\qquad +\delta_{y_{n-1}-\epsilon \sigma(y_{n-1})}(B)\bigr) \cdots\biggr)  d\gamma[y_{1}](y_{2})\biggr) \!d\gamma[x_{0}](y_{1})\\
	&=\int_{A_{1}}\left(\int_{A_{2}}\left(
	\cdots  G(y_{1}, \ldots, {y_{n-1}})
	\right) \cdots d\gamma[y_{1}](y_{2})\right) d\gamma[x_{0}](y_{1})\\
	\end{align*}
	for all cylinders $A=A_{1}\times\cdots\times A_{n-1}$. Thus, we must have
	\begin{align*}
	& G(y_{1}, \ldots,{y_{n-1}})\\
	&= q
	\frac{|B_{\epsilon}(y_{n-1})\cap B|}{|B_{\epsilon}(y_{n-1})|} + \frac{1-q}{2}\,(\delta_{y_{n-1}+\epsilon \sigma(y_{n-1})}(B)+\delta_{y_{n-1}-\epsilon \sigma(y_{n-1})}(B))\\
	&= \M^\eps_{\sigma(y_{n-1})}\chi_B(y_{n-1}).
	\end{align*}
\end{proof}

\subsection{Proof of Lemma \ref{martingale1}. ($|x_{n}|^{2}-n c_{N,p}  \epsilon^{2}$ is a martingale).}\label{proofof24}

\begin{proof}
	The dominative $p$-Laplacian of the paraboloid $\phi(x) := \vert x\vert^2$ is $\D_p\phi(x) = 2(N+p-2)$ and $\M^\eps_\xi\phi = \M^\eps\phi$ for every $\xi\in\mathbb{S}^{N-1}$.
	It follows from
	Lemma \ref{key} and \eqref{quadid} that
	\begin{align*}
	\mathbb{E}^{x_{0}}_\sigma\left[ |x_{n}|^{2}  \,|\, \mathcal{F}^{x_{0}}_{n-1}\right] (x_{n-1}) = \M^\eps_{\sigma(x_{n-1})}\phi(x_{n-1})
	&= \M^\eps\phi(x_{n-1})\\
	&= \phi(x_{n-1}) + \eps^2\frac{\D_p\phi}{C_{N,p}}\\
	&= \vert x_{n-1}\vert^2 + \eps^2c_{N,p},
	\end{align*}
	where $c_{N,p} := \frac{N+p-2}{N+p}$. Thus
	\begin{align*}
	\mathbb{E}^{x_{0}}_\sigma\left[|x_{n}|^{2}-n c_{N,p}  \epsilon^{2}  \,|\, \mathcal{F}^{x_{0}}_{n-1}\right] (x_{n-1})
	&= |x_{n-1}|^{2}+ c_{N,p}  \epsilon^{2}- n c_{N,p}  \epsilon^{2}\\
	&= |x_{n-1}|^{2}- (n-1) c_{N,p}  \epsilon^{2}
	\end{align*}
for every control $\sigma$.
\end{proof}

\subsection{Proof of Theorem \ref{meanvalue=stochastic}. (Equality of the stochastic solution $u_\eps$ and the Mean Value solution $v_\eps$)}\label{proofofthm21}

\begin{proof}
Certainly, if $x_0\in\Gamma_\eps$, then $\tau_\sigma = 0$ for all controls $\sigma$ and thus
\[u_\eps^\sigma(x_0) = \mathbb{E}^{x_{0}}_{\sigma}\left[F(x_{\tau_\sigma})\right] =  \mathbb{E}^{x_{0}}_{\sigma}\left[F(x_0)\right] = F(x_0).\]
If $x_0\in\oO$, then by Lemma \ref{key}, 
\[\mathbb{E}^{x_{0}}_{\sigma}\left[ v_{\epsilon}(x_n)\,|\, \mathcal{F}^{x_0}_{n-1}\right] (x_{n-1}) = \M^\eps_{\sigma(x_{n-1})} v_\eps (x_{n-1}) \leq \M^\eps v_\eps(x_{n-1}) = v_\eps(x_{n-1})\]
and $\{v_{\epsilon}\circ \mathbf{x_{n}}\}_{n\ge1}$ is a supermartingale with respect to the filtration
$\left\{ \mathcal{F}^{x_{0}}_{n}\right\}_{n\ge 1}$ for all controls $\sigma$. 

We use now Doob's theorem for supermartingales to move from the boundary back to $x_{0}$:
\begin{align*}
u_\eps(x_0)
	&= \sup_{\sigma}\left(\mathbb{E}^{x_{0}}_{\sigma}\left[  F(x_{\tau})\right]\right)\\
	&= \sup_{\sigma}\left(\mathbb{E}^{x_{0}}_{\sigma}\left[  v_{\epsilon}(x_{\tau})\right]\right)\\
	&\le \sup_{\sigma}\left(\mathbb{E}^{x_{0}}_{\sigma}\left[  v_{\epsilon}(x_{0})\right]\right)\\
	&= v_{\epsilon}(x_{0}).
\end{align*}
 
To show the opposite inequality,  we proceed by contradiction. Suppose that  there exists $x_{0}\in \Omega$ and $\eta>0$ such that $$u_\eps(x_0) \ge \eta+ v_{\epsilon}(x_{0}).$$
Choose a strategy $\sigma_{0}$ such that 
$$\mathbb{E}^{x_{0}}_{\sigma_{0}}\left[  F(x_{\tau})\right]\ge u_\eps(x_0) -\eta/2 \ge v_{\epsilon}(x_{0})+\eta/2. $$
We use again the fact that $v_{\epsilon}(x_{n})$ is a supermartingale with respect to any strategy and that
$v_{\epsilon}(y)=F(y)$ for $y\in \Gamma_{\eps}$ to deduce
$$ v_{\epsilon}(x_{0})\ge \mathbb{E}^{x_{0}}_{\sigma_{0}}\left[  v_{\epsilon}(x_{\tau})\right]=\mathbb{E}^{x_{0}}_{\sigma_{0}}\left[   F(x_{\tau})\right]\ge v_{\epsilon}(x_{0})+\eta/2,$$ which is clearly a contradiction.

\end{proof}

\subsection{Proof of Lemma \ref{boundaryestimate}}\label{proofof26}

The strategy to prove this lemma is as follows. First, we prove the convergence for smooth functions as done in \cite{PS08} for the $p$-Laplacian for functions  with non-vanishing gradient. We apply this result to the radial barriers which are translations and scaling of the fundamental solution, and then iterate following the argument of \cite{MPR12} for $p$-harmonic functions. 
\par
Consider the case of smooth functions $w\in C^3(\Omega_\eps)$
satisfying $\D_p w=0$ in the interior of $\Omega_\eps$.
Since the function $w$ is continuous we can apply the Dubins-Savage selection theorem (Theorem 5.3.1 in \cite{S98}) 
to deduce the existence of an optimal Borel strategy $\sigma_{0}$ such that
$$\M^\eps w(x) = \M_{\sigma_{0}(x)}^\eps w(x). $$ 
Note that  from the expansion
\eqref{smoothexpansion} we have, uniformly in 
$\oO$ that 
\begin{equation}\label{uniformexpansion}
w(x)=\M^\eps w(x)+O(\epsilon^3).
\end{equation}
%Fix a control  $\sigma$ and run the corresponding  process $x_0, x_1, \ldots$ From estimate 
%\eqref{uniformexpansion} we get

\begin{lemma}\label{subsuper} There exists a constant $C_{1}>0$ that depends on $v$ and $\Omega$ but it is independent of $\epsilon>0$,  such that:\par
 (i) For an arbitrary control $\sigma$ the sequence of random variables
$$M_k=w(x_k)-C_1 k \epsilon^3$$ is a \textsc{supermartingale}.\par
(ii) For the control $\sigma_{0}$ defined above the sequence of random variables
$$N_k=w(x_k)+C_1 k \epsilon^3$$ is a \textsc{submartingale}
\end{lemma}
\begin{proof}
We choose $C_{1}$ given by \eqref{uniformexpansion} and calculate:
\begin{align*}
\mathbb{E}^{x_{0}}_{\sigma}\left[ M_{k}\,|\, \mathcal{F}^{x_{0}}_{k-1}
\right]
	&= \mathbb{E}^{x_{0}}_{\sigma}\left[ w(x_k)\,|\, \mathcal{F}^{x_{0}}_{k-1}
\right]- C_1 k \epsilon^3\\
	&= \M^\eps_{\sigma(x_{k-1})}w(x_{k-1}) - C_1 k \epsilon^3\\
	&\le \M^\eps w(x_{k-1}) - C_1 k \epsilon^3\\
	&\le w(x_{k-1})+ C_{1}\epsilon^{3}- C_1 k \epsilon^3\\
	&= w(x_{k-1})- C_{1}(k-1)\epsilon^{3}\\
	&= M_{k-1}.
\end{align*}
\begin{align*}
\mathbb{E}^{x_{0}}_{\sigma_{0}}\left[ N_{k}\,|\, \mathcal{F}^{x_{0}}_{k-1}
\right]
&= \mathbb{E}^{x_{0}}_{\sigma_{0}}\left[ w(x_k)\,|\, \mathcal{F}^{x_{0}}_{k-1}
\right] + C_1 k \epsilon^3\\
&= \M^\eps_{\sigma_{0}(x_{k-1})}w(x_{k-1}) + C_1 k \epsilon^3\\
&= \M^\eps w(x_{k-1}) + C_1 k \epsilon^3\\
&\ge w(x_{k-1}) - C_{1}\epsilon^{3} + C_1 k \epsilon^3\\
&= w(x_{k-1}) + C_{1}(k-1)\epsilon^{3}\\
&= N_{k-1}.
\end{align*}
\end{proof}

Let $w_\eps$ be the mean value solution with boundary values equal to $w$. That is, $\M^\eps w_\eps = w_\eps$ in $\oO$ and $w_\eps = w$ on $\Gamma_\eps$.

\begin{corollary}\label{smoothcase}
There exists a constant $C_{2}>0$ depending on $w$ and $\Omega$ but independent of
$\epsilon$ such that  for all $x\in\overline{\Omega}$ we have
$$|w(x)-w_\epsilon(x)|\le C_2 \epsilon$$
\end{corollary}

\begin{proof} 
From Theorem  \ref{meanvalue=stochastic} and Lemma \ref{subsuper} (i) we have
\begin{align*}
w_\epsilon(x_0)
	&= \sup_\sigma\left(\mathbb{E}_\sigma^{x_0}[w(x_{\tau_\sigma})]\right)\\
	&= \sup_\sigma\left(\mathbb{E}_\sigma^{x_0}[w(x_{\tau_\sigma})-C_1\tau_\sigma \epsilon^3+
C_1\tau_\sigma \epsilon^3]\right)\\
	&\le \sup_\sigma\left(\mathbb{E}_\sigma^{x_0}[w(x_{\tau_\sigma})-C_1\tau_\sigma \epsilon^3]\right)+\sup_\sigma\left(
\mathbb{E}_\sigma^{x_0}[C_1\tau_\sigma \epsilon^3]\right)\\
	&\le w(x_0) + C_1 \epsilon^3 \sup_\sigma\left(\mathbb{E}_\sigma^{x_0}[\tau_\sigma ]\right),
\end{align*}
and from Lemma \ref{subsuper} (ii) we have
\begin{align*}
w_\epsilon(x_0)
	&= \sup_\sigma\left(\mathbb{E}_\sigma^{x_0}[w(x_{\tau_\sigma})]\right)\\
	&\ge \left(\mathbb{E}_{\sigma_0}^{x_0}[w(x_{\tau_\sigma})+C_1\tau_\sigma \epsilon^3-C_1\tau_\sigma \epsilon^3]\right)\\
	&= \mathbb{E}_{\sigma_0}^{x_0}[w(x_{\tau_{\sigma_0}})+C_1\tau_{\sigma_0} \epsilon^3]-\mathbb{E}_{\sigma_0}^{x_0}[C_1\tau_{\sigma_0} \epsilon^3]\\
	&\ge w(x_0) - C_1 \epsilon^3 \sup_\sigma\left(\mathbb{E}_\sigma^{x_0}[\tau_\sigma ]\right).
\end{align*}
Therefore,
$$|w(x)-w_\epsilon(x)|\le C_1 \epsilon^3 \sup_\sigma\left(
\mathbb{E}_\sigma^{x_0}[\tau_\sigma ]\right)\le C_{1}C(\Omega, N)\, \epsilon$$
by the stopping time bound (\ref{taubound}).
\end{proof}

We also give an alternative simpler proof of Corollary \ref{smoothcase} that do not rely on  Lemma \ref{subsuper} nor use any selection theorems. 

\begin{proof}[Analytic proof of Corollary \ref{smoothcase}]
Fix $\eps'>0$ and choose a ball with radius $R_\Omega>0$ and centre $x_*$ so that $\Omega_{\eps'}\subseteq B_{R_\Omega}(x_*)$. Let $\phi$ be the paraboloid $\phi(x) := \frac{C_1}{c_{N,p}}\vert x-x_*\vert^2$, and for $0<\eps\leq\eps'$ define
\[h_\eps(x) := w(x) - w_\eps(x) + \eps\phi(x).\]
Here, $c_{N,p} = \frac{N+p-2}{N+p}$ and $C_1>0$ is such that $\vert \M^\eps w - w\vert\leq C_1\eps^3$ in $\oO$. Then
\[\M^\eps\phi(x) = \frac{\eps^2}{C_{N,p}}\D_p\phi +\phi(x) = C_1\eps^2 +\phi(x),\]
so
\begin{align*}
\M^\eps h_\eps &= \M^\eps[w-w_\eps] + \eps\M^\eps\phi\\
&\geq \M^\eps w - \M^\eps w_\eps + C_1\eps^3 + \eps\phi\\
&\geq w - C_1\eps^3 - w_\eps + C_1\eps^3 + \eps\phi\\
&= h_\eps,
\end{align*}
and hence $h_\eps$ is an $\eps$-mean value subsolution. Use $\phi\geq 0$ and the maximum principle to obtain that $w - w_\eps \leq h_\eps$ and $h_\eps\vert_{\Gamma_\eps} = \eps\phi\vert_{\Gamma_\eps} \leq \frac{C_1}{c_{N,p}}R_\Omega^2 \eps$. The same analysis works for the function $g_\eps := w_\eps - w + \eps\phi$, and thus
\[\vert w(x) - w_\eps(x)\vert \leq \frac{C_1}{c_{N,p}}R_\Omega^2\eps.\]
\end{proof}

Next, we adapt the argument used in  \cite{MPR12}  for $p$-harmonious functions.
First, we construct  upper barriers. Consider the ring domain $B_{R}(x_{0})\setminus \overline{B_{r}(x_{0})}$ and assign boundary values
$m$ on the inner boundary $|x-x_{0}|=r$ and $M$ on the outer boundary $|x-x_{0}|=R$ satisfying $m\le M$. 
Set $b=-(N+p-4)$. If $b=0$, then we  must have $N=p=2$ since $N\ge2$ and $p\ge2$.  In this case, we define
\begin{equation}\label{b=0}
U(x)=\frac{M-m}{\log(R/r)} \log\left(\frac{|x-x_{0}|}{r}\right)+m.
\end{equation}
When $b<0$ we set instead
\begin{equation}\label{b<0}
U(x)=\frac{M-m}{R^{b}-r^{b}} \left(|x-x_{0}|^{b}-{r}^{b}\right)+m.
\end{equation}
In each case we have $\D_{p}U=0$ in $B_{R}(x_{0})\setminus \overline{B_{r}(x_{0})}$ with boundary values
$m$ on the inner boundary $|x-x_{0}|=r$ and $M$ on the outer boundary $|x-x_{0}|=R$.
\par
Since $\Omega$ is Lipschitz, it is clear that $\Omega$ satisfies the following regularity condition:
\begin{equation*}
\begin{split}
&\text{There exists} \,  \bar{\delta}>0 \, \text{and} \,  \mu\in(0,1) \,  \text{such that for every} \,  \delta\in(0,\bar{\delta}) \,  \text{and} \,  y\in\partial\Omega \, \\
&\text{there exists a ball}  \, B_{\mu\delta}(z) \,  \text{strictly contained in} \, B_\delta(y)\setminus\Omega.
\end{split}
\end{equation*}
Let $u_{\epsilon}$  be as in Lemma \ref{boundaryestimate}.
Fix $\delta\in(0,\bar{\delta})$.
For $y\in\partial\Omega$ consider: 
\begin{equation}\label{Beps}
m^\eps(y) :=\sup_{B_{5\delta}(y)\cap \Gamma^\eps}F  \quad
\text{and}\quad 
M^\eps :=\sup_{ \Gamma^\eps}F.
\end{equation}
Let $\theta\in (0,1)$ depending only on $\mu$, $N$ and $p$ to be determined later. Set
 $\delta_{k}=\delta/4^{k-1}$ for $k\ge 0$ and  define
\begin{equation}\label{Meps}
M_k^\eps(y)=m^\eps(y)+\theta^k(M^\eps-m^\eps(y)).
\end{equation}
By the regularity assumption on $\Omega$, there exist balls $B_{\mu \delta_{k+1}}(z_k)$ contained in $B_{\delta_{k+1}}(y)\setminus \Omega$ for all $k\in\mathbb{N}$. Note that $\mu$ is independent of $k$ and $\delta$.
The iteration lemma is the following:
\begin{lemma}\label{Lemma induction k} There exists  $\theta\in (0,1)$ depending only on $\mu$, $N$ and $p$ such that the following holds: 
Fix $\eta>0$ and let $y\in\partial\Omega$ and $\eps_k>0$. Under the above notations, suppose that for all $\eps<\eps_k$ we have:
\begin{equation*}
u_{\epsilon}\leq M_k^\eps(y) \quad \text{in}\quad B_{\delta_k}(y)\cap \Omega. 
\end{equation*}
Then, either $M_k^\eps(y)-m^\eps(y)\leq \frac{\eta}{4}$ or there exists $\eps_{k+1}=\eps_{k+1}(\eta, \mu,\delta,N,p,G)\in(0,\eps_k)$  such that:
\begin{equation*}
u_\eps \leq M_{k+1}^\eps(y) \quad \text{in}\quad B_{\delta_{k+1}}(y)\cap \Omega
\end{equation*}
for all $\eps\leq \eps_{k+1}$. 
\end{lemma}
\begin{proof} We will present the case $b<0$. 
Suppose that we are in the case $M_k^\eps(y)-m^\eps(y)>\frac{\eta}{4}$.
For notational convenience set  $m=m^\eps(y)$ and $M_k=M_k^\eps(y)$.
Consider the barrier $U_{k}$ defined on the ring $R_{k}=B_{\delta_k}(z_k)\setminus \overline{B_{\mu\delta_{k+1}}(z_k)} $
\begin{equation*}
U_{k}(x)=\frac{M_{k}-m}{\delta_{k}^{b}-(\mu\delta_{k+1})^{b}} \left(
|x-z_k|^b-(\mu\delta_{k+1})^b
\right)+m.
\end{equation*}
Note that  $U_k$ is increasing in $|x-z_{k}|$ is smooth  and solves the problem:
\begin{equation*}
\begin{cases}
\D_{p} (U_k) = 0 &\quad \text{in}\quad B_{\delta_k}(z_k)\setminus \overline{B_{\mu\delta_{k+1}}(z_k)} \\
U_k = m &\quad \text{on}\quad \partial B_{\mu\delta_{k+1}}(z_k)\\
U_k = M_k &\quad \text{on}\quad \partial B_{\delta_{k}}(z_k).
\end{cases}
\end{equation*}
We will establish several upper bounds for $\eps_{k+1}$, and   take $\eps_{k+1}$ to be the minimum of such bounds. \par

First, let $\eps_{k+1}=\frac{\mu\delta_{k+1}}{2}$.
For $\eps\leq \eps_{k+1}$, extend the barrier $U_k$ to the ring $$R_{k,\eps}= B_{\delta_k+2\eps}(z_k)\setminus \overline{B_{\mu\delta_{k+1}-2\eps}(z_k)}.$$
Let $U^\eps_k$ be $\epsilon$-mean value solution in $R_{k}=B_{\delta_k}(z_k)\setminus\overline{B_{\mu\delta_{k+1}}(z_k)}$ with boundary values $U_k$ on $R_{k,\eps}\setminus R_{k}$, the outer $\eps$-neighborhood of $R_{k}$. Since $R_{k}$ is a smooth domain, by 
Corollary \ref{smoothcase}  we have that $U^\eps_k$ converges to $U_k$ uniformly in $\tilde{X}$ as $\eps\to 0$. 
%Let $\Gamma^\delta=\{p\in\mathbb{H}\setminus\Omega \, |\, d(p,\partial\Omega)\leq \delta\}$. 
%Note that if $M_k=B$, then the result of the Lemma is trivial. Otherwise, 
Hence, given $$\gamma=\frac{(1/2)^{b}-((2-\mu)/4)^{b}}{8}\eta,$$
there exists $\eps_{k+1}=\eps_{k+1}(\gamma)>0$ such that: 
$$|U^\eps_k-U_k|\leq \gamma$$
for $\eps\leq \eps_{k+1}$ and for every $p\in\tilde{X}$.\par
Next, define \begin{equation*}
\alpha = \frac{1-(1/2)^{b}}{1-(\mu/4)^{b}}\quad \text{and}\quad 
\beta = \frac{(1/2)^{b}-(\mu/4)^{b}}{1-(\mu/4)^{b}}
\end{equation*}
and note that $\alpha$ and $\beta$ are non-negative and that $\alpha+\beta=1$. 

We now prove the following claim:
\begin{claim}\label{claim}
$$\alpha v^\eps +\beta m\leq U_k+\gamma \quad \text{in}\quad B_{\delta_k/2}(z_k)\cap\Omega,$$
for $\eps\leq \eps_{k+1}$. 
\end{claim}
From the comparison principle (Lemma \ref{comparisonprinciple}) we get  $$\partial_\eps(B_{\delta_k/2}(z_k)\cap\Omega)\subseteq\Gamma_1^\eps\cup \Gamma_2^\eps,$$
where  $\Gamma_1^\eps=B_{\delta_k/2+\eps}(z_k)\cap\Gamma^\eps$ and $\Gamma_2^\eps=(B_{\delta_k/2+\eps}(z_k)\setminus \overline{B_{\delta_k/2}(z_k)})\cap\Omega$.\par
On $\Gamma_1^\eps$, we have $u^\eps=F\leq m$, hence:
$\alpha u^\eps+\beta m\leq m=\inf_{R_{k}} U_k\leq U_k\leq U_k^\eps+\gamma,$
since $\Gamma_1^\eps\subset R_{k,\eps}$.\par
On $\Gamma_2^\eps$, we have $v^\eps\leq M_k$ by assumption, because $B_{\delta_k/2+\eps}(z_k)\subset B_{\delta_k}(y)$. 
For $x\in\partial B_{\delta_k/2}(z_k)$, we have $|x-z_{k}|=\delta_k/2$, hence:
\begin{equation}
\begin{split}
U_k(x)& = \frac{M_{k}-m}{\delta_{k}^{b}-(\mu\delta_{k+1})^{b}} \left(
|\delta_{k}/2|^b-(\mu\delta_{k+1})^b
\right)+m
\\
& =  \frac{M_{k}-m}{1-(\mu/4)^{b}} \left(
(1/2)^b-(\mu/4)^b\right)+m
\\& =  \frac{1-(1/2)^{b}}{1-(\mu/4)^{b}} m + \frac{(1/2)^{b}-(\mu/4)^{b}}{1-(\mu/4)^{b}}M_{k}
\\& =  \alpha m + \beta M_{k}
\end{split}
\end{equation}
and by monotonicity
$U_k\geq \alpha m+\beta M_{k}$ in $\Gamma_2^\eps$.
Hence:
\begin{equation*}\label{boundary2}
\alpha m +\beta v_{\eps}\leq \alpha m +\beta M_{k}\leq U_k\leq U_k^\eps+\gamma
\end{equation*}
in $\Gamma_2^\eps$.
%Combining with \eqref{boundary2} we get \eqref{claim}.
In conclusion, we have: 
$\alpha m +\beta v_{\eps}\leq U_k^\eps+\gamma \quad\text{in}\quad \partial_\eps(B_{\delta_k/2}(z_k)\cap\Omega),$
and the claim follows again by the comparison principle Lemma \ref{comparisonprinciple}.

\medskip

Consider next the intersection  $B_{\delta_{k+1}}(y)\cap\Omega$. We have $B_{\delta_{k+1}}(y) \subset B_{(2-\mu)\delta_{k+1}}(z_k)$ and for $x\in B_{(2-\mu)\delta_{k+1}}(z_k)$ we have:
\begin{equation}\label{improved0}
\begin{split}
U_k(x)&\leq  \frac{M_{k}-m}{\delta_{k}^{b}-(\mu\delta_{k+1})^{b}} \left(
((2-\mu)\delta_{k+1})^b-(\mu\delta_{k+1})^b
\right)+m\\
&=  \frac{M_{k}-m}{1-(\mu/4)^{b}} \left(
((2-\mu)/4)^b-(\mu/4)^b
\right)+m\\
&=\alpha' m+\beta'  M_{k},
\end{split}
\end{equation}
where we have set
$$  \alpha' = \frac{1-((2-\mu)/4)^b}{1-(\mu/4)^b}  \quad \text{and}\quad \beta' = \frac{(2-\mu)/4)^b-(\mu/4)^b}{1-(\mu/4)^b} .$$
Also, note that 
$B_{\delta_{k+1}}(y)\subset B_{\delta_k/2}(z_k)$,
hence by \eqref{claim} we get:
\begin{equation}\label{improved}
\alpha m+\beta v^{\eps}\leq U_k+\gamma \quad\text{in}\quad B_{\delta_{k+1}}(y)\cap\Omega.
\end{equation}  
Combining \eqref{improved0} and \eqref{improved}, for $p\in B_{\delta_{k+1}}(y)\cap\Omega$ and $\eps<\eps_{k+1}$, we get:
\begin{equation*}
\begin{split}
v^\eps(p)&\leq 
\frac{\alpha ' -\alpha}{\beta}m+\frac{\beta '}{\beta}M_k+\frac{\gamma}{\beta}\\
&= m+\frac{\beta '}{\beta}(M_k-m)+\frac{\gamma}{\beta}.
\end{split}
\end{equation*}
Observe that $\beta'/\beta\in (0,1)$ and that $\beta'<\beta$. Recall that we have chosen  
$$\gamma=\frac{(1/2)^{b}-((2-\mu)/4)^{b}}{8}\eta \le  \frac{(1/2)^{b}-((2-\mu)/4)^{b}}{2} (M_{k}-m).$$
Thus, we get 
\begin{equation*}
\begin{split}
v^\eps(p)
&\le  m+\left(\frac{\beta' }{\beta}+\frac{(1/2)^{b}-((2-\mu)/4)^{b}}{2\beta}\right)(M_k-m), 
\end{split}
\end{equation*}
and setting 
\begin{equation}\label{theta}
\theta= \frac{\beta' }{\beta}+\frac{(1/2)^{b}-((2-\mu)/4)^{b}}{2\beta}
\end{equation} we get
$$v^\eps(p)
\le  m+\theta(M_k-m)\le m +\theta^{k+1}(M^{\epsilon}-m).$$
\end{proof}

%\color{red}
%(How about the lower barrier? In principle, that should be easier. But it has to be stated and proved.)\par
%\normalcolor
%\color{blue} Yes,  see below. \normalcolor

%Remember that $-C_1\vert x-x_0\vert^{- (N+p-4)} + C_2$ is a solution whenever $C_1\geq 0$ (as in the case \eqref{b<0} because $R^b < r^b$). It is the \textbf{only} radially \textbf{increasing} solution. When $C_1 < 0$, it is a radially decreasing subsolution, which perhaps, is enough for us.

%The only radially \textbf{decreasing solution}, is the usual fundamental solution $-C_1\frac{p-1}{p-N}\vert x-x_0\vert^\frac{p-N}{p-1} + C_2$ \textbf{\underline{where $C_1\geq 0$}}.

The next Corollary, whose proof follows in a standard way from Lemma \ref{Lemma induction k},   implies one half of Lemma \ref{boundaryestimate}.\par
%\color{red}(I'm not sure how to modify the Corollary according to the new definition of $\Gamma_\eps$)\normalcolor\par
%\color{blue} Why do we need to modify it? Recall that $F$ is Lipschitz in $\partial
%\Omega$ and on $\Gamma_\eps$. \normalcolor\par
\begin{corollary}\label{TheoremBoundary}
Given $\eta>0$, there exist $\delta=\delta(\eta, F, \bar{\delta})$, $k_0=k_0(\eta, \mu, p, F)$, $\eps_0=\eps_0(\eta, \delta,\mu, k_0)$ such that:
\begin{equation*}
u^\eps(x)-F(y)\leq \frac{\eta}{2},
\end{equation*}
for all $y\in\partial\Omega$, $x\in B_{\delta/4^{k_0}}(y)\cap\overline{\Omega}$ and $\eps\leq\eps_0$.
\end{corollary}

To prove the (easier) lower bounds we observe that if $v$ is a harmonic function, then it is also a subsolution of the equation $\D_{p}u(x)  = 0$. Thererefore if $u$ and $v$ agree on the boundary of a domain, we must have $v\le u$ in the domain. Thus, lower bounds for $u$ follow from lower bounds for $v$.  This suggest building barriers using the fundamental solution of the Laplacian. Repeating the argument of the proof of Lemma \ref{Lemma induction k} 
with minima in lieu of maxima  and using the fundamental solution of the Laplacian as barriers, we get the the analogue of Lemma \ref{Lemma induction k}  for lower bounds, and the other half of Lemma \ref{boundaryestimate}.
\subsubsection{Proof of Lemma \ref{viscosity}}\label{proofof27}
Let us prove that $\overline{u}$ is a viscosity subsolution; that is, it satisfies $\D_{p} \overline{u} \ge 0$ in the viscosity sense. Let $x_{0}\in \Omega$ and choose $\phi\in C^{2}(\Omega)$ such that \textit{$\phi$ touches $\overline{u}$ from above at $x_{0}$}; i.e. we have $\overline{u}(x_{0})=\phi(x_{0})$ and
 $\overline{u}(x)< \phi(x)$ for  $x\in\Omega\setminus\{x_{0}\}$.
 The following proposition  is standard (see Lemma 4.2 in \cite{B94} and the Mathoverflow discussion\cite{MOF16}). We include the proof  for completeness. 
\begin{propo}\label{limitsofmaxima} Suppose that $\overline{B(x_{0},r)}\subset \Omega$. Then, there exists a sequence of numbers $\epsilon_{n}\to0$ and a sequence of  points $y_{n}\to x_{0}$ such that
 $u_{\epsilon_{n}}(y_{n})\to \overline{u}(x_{0})$ and  the function  $\phi- u_{\epsilon_{n}}$ has an approximate minimum in   $\overline{B(x_0,r)}$ at the point $y_{n}$; that is, we have:
 \begin{equation}\label{approxmin}
 \phi(y_{n})- u_{\epsilon_{n}}(y_{n}) \le \phi(y)- u_{\epsilon_{n}}(y) + \epsilon_{n}^{3}
 \end{equation}
 for all $y\in B(x_{0},r)$
 %,  and we have
 %$\phi(x_{n}) -u_{\epsilon_{n}}(x_{n})   \le  \phi(x) -u_{\epsilon_{n}}(x)$ for $x\in  \overline{B_{\eps_n}}(x_n)$ 
  \end{propo}
\begin{proof}
Choose sequences of numbers $\epsilon_{n}$ and points $x_{n}$ such that $\epsilon_{n}\to0$,  $x_{n}\to x_{0}$, and
$u_{\epsilon_{n}}(x_{n})\to\bar{u}(x_{0})$ as $n\to\infty$. Select a point $y_{n}\in \overline{B(x_{0},r)}$ such that
\begin{equation}\label{approxmin2}
\inf_{ y\in \overline{B(x_{0},r)}}\, \,  \phi(y)-u_{\epsilon_{n}}(y)
\ge  \phi(y_{n})-u_{\epsilon_{n}}(y_{n}))- \epsilon_{n}^{3}.
\end{equation}

Select convergent subsequences $(\epsilon_{n}, x_{n}, y_{n})\to (0, x_{0}, y_{0})$ that we relabel with the index $n$ again.
We have, using the defintion of $\bar{u}(x_{0})$, that
\[
\begin{array}{rcccl} 0 & = & \phi(x_{0})-\bar{u}(x_{0}) &=& \lim_{n\to\infty} \phi(x_{n})-u_{\epsilon_{n}}(x_{n})
\\
& & &\ge & \liminf_{n\to\infty} \phi(y_{n})-u_{\epsilon_{n}}(y_{n})- \epsilon_{n}^{3}\\
& & &\ge & \liminf_{\epsilon\to0, y\to \hat{y}}\phi(y)-u_{\epsilon}(y)- \epsilon^{3}\\
& & &=  &\phi(y_0) - \bar{u}(y_0),
\end{array}
\]
which would be positive, unless we have $y_0=x_{0}$. 
The proposition then follows from \eqref{approxmin2}.
\end{proof}
To continue proving Lemma \ref{viscosity} 
start with  $$\phi(y_{n}) -u_{\epsilon_{n}}(y_{n})   \le  \phi(x) -u_{\epsilon_{n}}(x)+ \epsilon_{n}^{3}$$
  and for $n$ large, noting that  $B(y_{n}, \epsilon_n)\subset B(x_0,r)$, integrate over the ball $B(y_{n}, \epsilon_n)$ to get 
\[
\begin{array}{rcl}
\phi(y_{n}) -u_{\epsilon_{n}}(y_{n}) &  \le &  \M^{\epsilon_{n}}_\xi\left[ \phi-u_{\epsilon_{n}} \right](y_{n})+ \epsilon_{n}^{3}\\
&=& \M^{\epsilon_{n}}_\xi\phi(y_{n}) - \M^{\epsilon_{n}}_\xi u_{\epsilon_{n}}(y_{n})+ \epsilon_{n}^{3}
\end{array}
\]
Therefore, we have
\[
\begin{array}{rcl}
\phi(y_{n}) -u_{\eps_{n}}(y_{n}) +\M^{\epsilon_{n}}_\xi u_{\eps_{n}}(y_{n}) &  \le & \M^{\epsilon_{n}}_\xi\phi(y_{n}) + \epsilon_{n}^{3},\\
\end{array}
\]
and taking supremum among all $|\xi|=1$ we get
$$
\phi(y_{n}) -u_{\epsilon_{n}}(y_{n}) +\M^{\epsilon_{n}} u_{\epsilon_{n}}(y_{n})   \le  \M^{\epsilon_{n}}\phi(y_{n})+ \epsilon_{n}^{3}
$$
from which we, using the fact that 
$u_{\epsilon_{n}}(y_{n}) =\M^{\epsilon_{n}} u_{\eps_{n}}(y_{n}),$
conclude that
\[
\begin{array}{rcl}
\phi(y_{n})  &  \le & \M^{\epsilon_{n}}\phi(y_{n})\\
&= & \phi(y_{n}) +\frac{\epsilon_n^2}{C_{N,p}}\, \D_{p}\phi(y_{n})+ o(\epsilon_{n}^{2})+ \epsilon_{n}^{3}.
\end{array}
\]
Therefore, letting $n\to\infty$,  we have   $\D_{p}\phi(x_{0})\ge 0$ and thus
  $\D_p \overline{u}(x_{0})\ge 0$ in the viscosity sense.
\par
A similar proof shows that  $\underline{u}$ is a viscosity supersolution. 
\subsubsection{Proof of Theorem \ref{main}} \label{proofofthm22} Given the boundary estimates
(\ref{upperbound}) and (\ref{lowerbound}),  we use the comparison principle for viscosity solution of $\D_{p}u = 0$ to conclude that
$$\overline{u }=\underline{u }$$
and 
$$\lim_{\epsilon\to0}u_{\epsilon}=\overline{u}=\underline{u}=u$$ uniformly in $\overline{\Omega}$,
where $u$ is  the unique solution to the Dirichlet problem (\ref{introdirichlet}).

\textbf{Acknowledgements:} Supported by the Norwegian Research Council (grant 250070), the Academy of Finland (grant SA13316965), and Aalto University. We thank Eero Ruosteenoja for pointing out a flaw in
our original manuscript.

 \newpage

\bibliographystyle{alpha}
\bibliography{ManfrediOct2017}

\end{document}